\documentclass[12pt,reqno]{amsart}
\usepackage{amsmath, amssymb, amsthm}
\usepackage{url}
\usepackage[breaklinks]{hyperref}

\renewcommand{\Im}{\operatorname{Im}}

\renewcommand{\Re}{\operatorname{Re}}
\renewcommand{\Im}{\operatorname{Im}}

\renewcommand{\(}{\left\(}
\renewcommand{\)}{\right\)}
\renewcommand{\[}{\left\[}
\renewcommand{\]}{\right\]}

\numberwithin{equation}{section}
\theoremstyle{plain}
\newtheorem{theorem}{Theorem}[section]
\newtheorem{lemma}[theorem]{Lemma}
\newtheorem{remark}[]{Remark}

\newtheorem{corollary}[theorem] {Corollary}

\makeatletter
\def\proof{\@ifnextchar[{\@oproof}{\@nproof}}
\def\@oproof[#1][#2]{\trivlist\item[\hskip\labelsep\textit{#2 Proof of\
		#1.}~]\ignorespaces}
\def\@nproof{\trivlist\item[\hskip\labelsep\textit{Proof.}~]\ignorespaces}

\makeatother

\begin{document}
	\title[An asymptotic expansion for a twisted Lambert series]{An asymptotic expansion for a twisted Lambert series associated to a cusp form and the M\"{o}bius function: level aspect}  
	
	 \author{Bibekananda Maji}
\address{Bibekananda Maji\\ Department of Mathematics \\
Indian Institute of Technology Indore \\
Indore, Simrol, Madhya Pradesh 453552, India.} 
\email{bibekanandamaji@iiti.ac.in}

\author{Sumukha Sathyanarayana}
\address{Sumukha Sathyanarayana\\Department of Mathematical and Computational Sciences\\ National Institute of Technology Karnataka, Suratkal \\
	Srinivasnagar, Mangalore 575025, Karnataka, India.} 
\email{neerugarsumukha@gmail.com\\ 
	sumukhas.177ma005@nitk.edu.in}

	\author{B. R.  Shankar }
	\address{Shankar B R\\Department of Mathematical and Computational Sciences\\ National Institute of Technology Karnataka, Suratkal \\
		Srinivasnagar, Mangalore 575025, Karnataka, India.}
	\email{shankarbr@gmail.com}

	
	\maketitle
	
	\begin{abstract}
Recently,  Juyal,  Maji and Sathyanarayana have studied a Lambert series associated with a cusp form over the full modular group and the M\"{o}bius function. In this paper, we 
investigate the Lambert series $\sum_{n=1}^{\infty}[a_f(n)\psi(n)*\mu(n)\psi'(n)]\exp(-ny),$ where $a_f(n)$ is the $n$th Fourier coefficient of a cusp form $f$ over any congruence subgroup, and $\psi$ and $ \psi'$ are primitive Dirichlet characters. This extends the earlier work to the case of higher level subgroups and also gives a character analogue.
	\end{abstract}

	\section{Introduction}
	Let ${a(n)}$ be a sequence of complex numbers. The Lambert series is defined by $ F(q)= \sum_{n=1}^{\infty} a(n)  \frac{q^n}{1-q^n} $ for $|q|<1$. By setting $b(n)=\sum_{d | n} a(d)$, we have $F(q)= \sum_{n=1}^{\infty} b(n) q^n$. Over the years, this family of Lambert series has been studied by many mathematicians \cite{Agarwal},  \cite{Berndt}. Hardy and Ramanujan \cite{HR} used the behaviour of certain Lambert series as $q \rightarrow 1$ to derive an asymptotic expansion for a general partition function. 
	
	 Let $\Delta(z):=\sum_{n=1}^{\infty} \tau(n) \exp(2 \pi i n z)$ be the Ramanujan cusp form of weight 12. In 1981, Zagier \cite[p.~417]{Zag92} predicted that the constant term of the automorphic form $y^{12} |\Delta(z)|^2$,  that is, the Lambert series 
	 $a_0(y):=y^{12} \sum_{n=1}^{\infty} \tau^2(n) \exp({-4 \pi ny})$
	 has an asymptotic expansion 
	of the form
	 \begin{equation} \label{zagier}
	  a_0(y)  \sim C+ \sum_{\rho} C_{\rho} y^{ 1 - \frac{\rho}{2}} \quad \textrm{as}\,\,  y \rightarrow 0^+, 
	 \end{equation} 	 
	 where $C$ and $C_\rho$ are some constants and the sum over $\rho$ runs through all non-trivial zeros of $\zeta(s)$. 
Further, he plotted the graph of $a_0(y)$ and found an oscillatory behaviour as $y \rightarrow 0^+$. Under the assumption of the Riemann hypothesis, \eqref{zagier} can be rewritten as
\begin{align*}
a_0(y) \sim C + y^{3/4} \sum_{n=1}^\infty a_n \cos\left( \phi_n + \frac{t_n\log(y)}{2}  \right), \quad {\rm as}\,\, y \rightarrow 0^{+}, 
\end{align*}
where $\rho = \frac{1}{2} \pm i t_n$ and $ a_n$, $\phi_n$ are some constants. 
This asymptotic expansion indicates that $a_0(y)$ has an oscillatory behaviour as $y \rightarrow 0^{+}$ due to the presence of cosine functions.  
Although this conjecture was overlooked by mathematicians for almost two decades,  it was eventually settled in 2000  by Hafner and Stopple \cite{HS}. Later, Chakraborthy et al. \cite{CKM} observed a similar phenomenon for any Hecke eigen form over the full modular group and they \cite{CJKM} generalized it for congruence subgroups.  

Recently, Juyal, Maji and Sathyanarayana \cite{JMS-2021} studied the Lambert series $\sum_{n=1}^\infty (a_f * \mu)(n) \exp(-ny)$, where $a_f(n)$ is the  $n$th Fourier coefficient of a cusp form $f$ of weight $k$ over $\textrm{SL}_2(\mathbb{Z})$ and $\mu(n)$ is the M\"{o}bius function. They derived the following exact formula in terms of the non-trivial zeros of the Riemann zeta function $\zeta(s)$. In \cite{JMS-2}, they have also studied a similar Lambert series associated with the symmetric square $L$-function. 
 \begin{theorem}\cite[Theorem 2.1]{JMS-2021}\label{JMS_main theorem}
Let $a_f(n)$ be the $n$th Fourier coefficient of the cusp form $f \in S_k(SL_2(\mathbb{Z}))$. Assume that all non-trivial zeros of $\zeta(s)$ are simple. Then for any positive $y$,
	\begin{align*}
\sum_{n=1}^{\infty}[a_f(n)*\mu (n)]e^{-ny}  & =    2\, \Gamma(k) \left(\frac{i}{2\pi}\right)^k \sum_{n=1}^{\infty}\frac{A_f^*(n)}{n^k}  \left[ {}_2F_1 \left(\frac{k}{2}, \frac{k+1}{2}; \frac{1}{2}; -\frac{y^2}{4n^2\pi^2} \right)-1\right] \\
& + \mathcal{P}(y), 
	\end{align*}
	where
	\begin{equation}\label{Af(n)}
A_f^{*}(n):= \left(a_f*\mu_{k}\right)(n), \quad {\rm with} \quad \mu_{k}(n)=\mu(n)n^{k-1},
\end{equation}
	and
	\begin{align}\label{P(y)}
\mathcal{P}(y) = 	\sum_{\rho}\frac{L(f,\rho)\Gamma(\rho)}{\zeta'(\rho)}\frac{1}{y^\rho},
	\end{align}
the sum over $\rho$ which runs through all non-trivial zeros of $\zeta(s)$, involves bracketing the terms so that the terms corresponding to $\rho_1$ and $\rho_2$ are included in the same bracket if they satisfy
	\begin{align}\label{bracketing}
|\Im(\rho_1) - \Im(\rho_2)| < \exp \left( -C \frac{\Im(\rho_1)}{\log(\Im(\rho_1) )} \right) + \exp \left( -C \frac{\Im(\rho_2)}{\log(\Im(\rho_2) )} \right),
\end{align}
where $C$ is some positive constant.
\end{theorem}
In this paper, we investigate the following Lambert series:
\begin{align}\label{main Lambert series}
A_f(y):= \sum_{n=1}^{\infty}[a_f(n)\psi(n)*\mu(n)\psi'(n)]\exp(-ny),
\end{align}
 where $a_f(n)$ is the  $n$th Fourier coefficient of a cusp form $f$ of weight $k$, level $Q$ and Nebentypus $\chi$, and $ \psi, \psi'$ are primitive Dirichlet characters. We derive an exact formula for the above Lambert series \eqref{main Lambert series} involving non-trivial zeros of $L(s,\psi')$ and a generalized hypergeometric function, and there by generalize Theorem \ref{JMS_main theorem} in two different directions.  On the one hand, our work generalize Theorem \ref{JMS_main theorem} for congruence subgroups and on the other hand, we also get a character analogue.  As an application, we also derive an asymptotic expansion and establish an oscillatory behavior of $y^{1/2} A_f(y)$ as $ y \rightarrow 0^{+}.$
\section{Preliminaries}
Let $k$ and $Q$ be two positive integers. Let $\chi $ 
be a Dirichlet character modulo $Q$ and the Gauss sum $\epsilon_{\chi}$  is defined by $\epsilon_{\chi} :=\sum_{j=1}^{Q} \chi(j) \exp\left({\frac{2 \pi i j}{Q}} \right)$.
 Define
\begin{align*}
\Gamma_{0}(Q):= \left\{  \left(\begin{array}{c c}  a & b \\
c & d  \end{array} \right) \in {\rm SL}_2(\mathbb{Z}) :    { c} \equiv 0 \pmod Q  \right\}.
\end{align*}
We denote $ S_k( \Gamma_{0}(Q), \chi)$ to be the space of cusp forms of weight $k$, level $Q$, and Nebentypus character $\chi$. Consider $f(z) \in  S_k( \Gamma_{0}(Q), \chi)$ with the Fourier series expansion
\begin{align}\label{Fourier}
f(z) = \sum_{n=1}^\infty a_{f}(n) \exp(2 \pi i n z), \quad \forall z \in \mathbb{H}.
\end{align}
The L-function associated to $f$ is defined by $L_f(s):= \sum_{n=1}^{\infty} \frac{a_{f}(n)}{n^s}$, which is absolutely convergent for $\Re(s)>\frac{k+1}{2}$. The analytic extension and functional equation for $L_f(s)$ was proved by Hecke.
It is known that for a positive integer $r$ such that $(Q,r)=1$ and  a primitive Dirichlet character $\psi$ modulo $r$, the $\psi$-twist of $f$ defined by 
$$
f_\psi(z)= \sum_{n=1}^\infty a_{f}(n) \psi(n) \exp(2 \pi i n z), 
$$
 is a element of $S_k( \Gamma_{0}(N), \chi \psi^2)$,  where $N=Qr^2$.
Hence it is natural to consider the following Dirichlet series:
 $$ L_f(s, \psi):= \sum_{n=1}^{\infty} \frac{a_{f}(n) \psi(n)}{n^s}, \quad \Re(s) > \frac{k+1}{2}.$$ 
 It is known as the $\psi$-twist of $L_f(s)$. We can analytically extend $L_f(s, \psi)$ into an entire function and the completed $L$-function $\Lambda_f(s, \psi) = \left(\frac{\sqrt{N}}{2 \pi} \right)^s \Gamma(s) L_f(s, \psi) $ satisfies the following functional equation \cite[p.~131]{Murty} 
 \begin{equation}\label{FE}
 \Lambda_f(s, \psi)  = \frac{i^k \epsilon_{\chi}^2 }{r} \chi(r) \psi(Q)  \Lambda_g(k-s, \overline{ \psi}) ,
 \end{equation}
 where $$g(z)=Q^{k/2} (Qz)^{-k} f\left(-\frac{1}{Qz}\right):=\sum_{n=1}^{\infty} a_g(n) \exp(2\pi i n z) \in S_k(\Gamma_{0}(Q), \overline{\chi} )).$$
Now we introduce two well-known special functions which are essential in this work.
Let $a_1, \cdots, a_p$ and $b_1, \cdots, b_q$ be complex numbers such that $b_j \ne 0, -1, -2,...$ for $ 1 \le j \le q$. The generalized hypergeometric series ${}_pF_q \left( a_1, \cdots, a_p;\, b_1, \cdots, b_q;\, z \right)$ \cite[p. 404, Equation 16.2.1]{NIST} is defined by
\begin{align}\label{hyper}
{}_pF_q \left( a_1, \cdots, a_p;\, b_1, \cdots, b_q;\, z \right):= \sum_{n=0}^{\infty} \frac{(a_1)_n \cdots (a_p)_n}{(b_1)_n\cdots (b_q)_n} \frac{z^n}{n!},
\end{align}
where $(a)_n:=\frac{\Gamma(a+n)}{\Gamma(a)}$. This series converges for all complex values of $z$ if $p \leq q$.  For $p=q+1$,  it converges for $|z|<1$ an further it can be analytically continued to the complex plane with a branch cut from $1$ to $+\infty$. 

Let $m,n,p,q$ be non-negative integers such that $0\leq m \leq q$, $0\leq n \leq p$. Let $a_1, \cdots, a_p$ and $b_1, \cdots, b_q$ be complex numbers such that $a_i - b_j \not\in \mathbb{N}$ for $1 \leq i \leq n$ and $1 \leq j \leq m$. Then the Meijer $G$-function \cite[p.~415, Definition 16.17]{NIST} is defined by the line integral:
\begin{align}\label{MeijerG}
G_{p,q}^{\,m,n} \!\left(  \,\begin{matrix} a_1,\cdots , a_p \\ b_1, \cdots , b_q \end{matrix} \; \Big| z   \right) := \frac{1}{2 \pi i} \int_L \frac{\prod_{j=1}^m \Gamma(b_j - s) \prod_{j=1}^n \Gamma(1 - a_j +s) z^s  } {\prod_{j=m+1}^q \Gamma(1 - b_j + s) \prod_{j=n+1}^p \Gamma(a_j - s)}\mathrm{ds}.
\end{align}
Here we note that the line of integration $L$, goes from $-i \infty$ to $+i \infty$,  separates the poles of the factors $\Gamma(1-a_j+s)$ from the factors  $\Gamma(b_j-s)$. This integral converges if $p+q < 2(m+n)$ and $|\arg(z)| < (m+n - \frac{p+q}{2}) \pi$. Slater's theorem \cite[p.~415, Equation 16.17.2]{NIST} gives a relation between Meijer $G$-function and generalized hypergeometric functions as follows:
If $p \leq q$, and $ b_j - b_k \not\in \mathbb{Z}$ for $j\neq k$, $1 \leq j, k \leq m$, then 
\begin{align}\label{Slater}
& G_{p,q}^{\,m,n} \!\left(   \,\begin{matrix} a_1, \cdots , a_p \\ b_1, \cdots , b_q \end{matrix} \; \Big| z   \right)  \\
& \quad = \sum_{k=1}^{m} A_{p,q,k}^{m,n}(z) {}_p F_{q-1} \left(  \begin{matrix}
1+b_k - a_1,\cdots, 1+ b_k - a_p \\
1+ b_k - b_1, \cdots, \star, \cdots, 1 + b_k - b_q 
\end{matrix} \Big| (-1)^{p-m-n} z  \right), \nonumber 
\end{align}
where $\star$ suggests that the term $1 + b_k - b_k$ is excluded and 
\begin{align*}
A_{p,q,k}^{m,n}(z) := \frac{ z^{b_k}  \prod_{ j=1,  j\neq k}^{m} \Gamma(b_j - b_k ) \prod_{j=1}^n   \Gamma( 1 + b_k -a_j ) }{ \prod_{j=m+1}^{q} \Gamma(1 + b_k - b_{j}) \prod_{j=n+1}^{p} \Gamma(a_{j} - b_k)  }.
\end{align*}
 \section{Statement of Results}
 For a fixed natural number $k$, we define $\mu_k(n) := \mu(n) n^{k-1}$.
 Let $\psi'$ be a primitive Dirichlet character modulo $M$, and   $$
 a:= 
 \begin{cases}
 0,              & \psi'(-1)=1,\\
 1,             & \psi'(-1)= -1.
 \end{cases}
 $$
The following theorem is the main result of this paper.
		\begin{theorem}\label{Main theorem2}
			Let $f \in S_k(\Gamma_{0}(Q),\chi)$ be a cusp form with the $n$th Fourier coefficient as $a_f(n)$.  Let $ \psi$ and $\psi'$ be primitive Dirichlet characters of modulus $r$ and $M$ respectively. Assume that all non-trivial zeros of $L(s,\psi')$ are simple.
			For any positive $y$,  we have
			\begin{align*}
			& \sum_{n=1}^{\infty}  [a_f(n)\psi(n)*\mu(n)\psi'(n)]\exp(-ny)  = 2 N^{k/2}  \Gamma(k+a)  \left( \frac{y N}{M}  \right)^a \left(\frac{i}{2\pi}\right)^{k+a} \frac{\chi(r) \psi(Q) \epsilon_\psi^2}{r\epsilon_{\psi'}}\\
			& \times \sum_{n=1}^{\infty}\frac{[a_g(n)\overline{\psi}(n)*\mu_k(n)\overline{\psi'}(n)]}{n^{k+a}}  
			 \left[ {}_2F_1 \left(\frac{k+a}{2}, \frac{k+1+a}{2}; \frac{1+2a }{2}; -\frac{N^2y^2}{4 M^2n^2\pi^2} \right)-(1-a) \right] \\
			& \hspace{7cm} +\mathcal{R}(y) + R_0,
			\end{align*}
			where $N= Q r^2$,  and the terms $R_0$ and $\mathcal{R}(y)$ are defined as
			\begin{align}\label{R(y)}
R_0= 
	 \begin{cases}
	\frac{L_f'(0, \psi)}{L'(0,\psi')},              &  {\rm if} \,  a=0 \, {\rm and} \, M>1, \\
	0,  & \text{otherwise},
	\end{cases} \quad	{\rm and} \quad					
			\mathcal{R}(y) = 	\sum_{\rho}\frac{L_f(\rho, \psi)\Gamma(\rho)}{L'(\rho, \psi')}\frac{1}{y^\rho},
			\end{align}
			where the sum over $\rho$ in $\mathcal{R}(y)$, runs through all non-trivial zeros of $L(s,\psi')$, involves bracketing the terms so that the terms corresponding to $\rho_1$ and $\rho_2$ are included in the same bracket if they satisfy
			\begin{align}\label{bracketing}
			|\Im(\rho_1) - \Im(\rho_2)| < \exp \left( - \frac{ C |\Im(\rho_1)|}{\log(|\Im(\rho_1) |+3)} \right) + \exp \left( - \frac{C |\Im(\rho_2)|}{\log(|\Im(\rho_2) |+3)} \right),
			\end{align}
			where $C$ is some positive constant.
		\end{theorem}
	
\begin{remark}
Theorem {\rm \ref{Main theorem2}} can be extended analytically for $\Re(y) >0. $ Also,  by substituting $Q=r=M=1$ in Theorem {\rm \ref{Main theorem2}},  one can immediately recover Theorem \ref{JMS_main theorem}.
\end{remark}
As a special case of Theorem \ref{Main theorem2}, by taking $M=r=1$, we get a higher level analogue of Theorem \ref{JMS_main theorem}. We note this special case as a corollary.
\begin{corollary}\label{M=r=1}
Let $f \in S_k(\Gamma_{0}(Q),\chi)$ be a cusp form.  Assume that all non-trivial zeros of $\zeta(s)$ are simple.
Then for $y >0$, we have
\begin{align*}
\sum_{n=1}^{\infty}[a_f(n)*\mu(n)]e^{-ny} & = \frac{i ^k \Gamma(k)Q^{\frac{k}{2}}  }{2^{k-1} \pi^k} 
 \sum_{n=1}^{\infty}\frac{[a_g(n)*\mu_k(n)]}{n^{k}}   \left[ {}_2F_1 \left(\frac{k}{2}, \frac{k+1}{2}; \frac{1}{2}; -\frac{Q^2y^2}{4n^2\pi^2} \right)-1 \right] \\
 & + \mathcal{R}(y), 
\end{align*}
where $\mathcal{R}(y)=\sum_{\rho}\frac{L_f(\rho)\Gamma(\rho)}{\zeta'(\rho)}\frac{1}{y^\rho}$, the sum over $\rho$ runs through all non-trivial zeros of $\zeta(s)$ involving bracketing as in \eqref{bracketing}.
\end{corollary}
On the other hand,  if we let $Q=1, M=r, $ and $\psi=\psi'$ in Theorem \ref{Main theorem2}, we get a character analogue of Theorem \ref{JMS_main theorem}.  Mainly,  we obtain the following identity.  
\begin{corollary}\label{Q=1,  M=r}
	Let $f \in S_k(\rm{SL}_2(\mathbb{Z}))$ be a cusp form.  Let $\psi$ be a primitive Dirichlet character modulo $r$. Assume that all non-trivial zeros of $L(s,\psi)$ are simple.   For  $y>0$, we have
\begin{align*}
\sum_{n=1}^{\infty}\psi(n)[a_f(n)*\mu(n)]\exp(-ny)= R_0+\mathcal{R}(y)+ 2 y^a r^{k+a-1} \epsilon_\psi  \left(\frac{i}{2\pi}\right)^{k+a} \\ 
\times \sum_{n=1}^{\infty}\frac{\overline{\psi}(n)[a_f(n)*\mu_k(n)]}{n^{k+a}}  \left[ {}_2F_1 \left(\frac{k+a}{2}, \frac{k+1+a}{2}; \frac{1+2a }{2}; -\frac{r^2y^2}{4n^2\pi^2} \right)-(1-a) \right],
\end{align*}
where $R_0$ is defined as in Theorem {\rm \ref{Main theorem2} } and  $\mathcal{R}(y) =\sum_{\rho}\frac{L_f(\rho, \psi)\Gamma(\rho)}{L'(\rho, \psi)}\frac{1}{y^\rho}$, where the sum over $\rho$ involves bracketing as in \eqref{bracketing}.
\end{corollary}

Now letting $Q=r=1$ in Theorem \ref{Main theorem2},  we obtain the following result.

\begin{corollary}\label{Q=r=1}
Let $f \in S_k(\rm{SL}_2(\mathbb{Z}))$ be a cusp form.  Let $\psi'$ be a primitive Dirichlet character modulo $M$.  Assume that all non-trivial zeros of $L(s,\psi')$ are simple.   Then
for any positive $y$,  we have
			\begin{align*}
			& \sum_{n=1}^{\infty}  [a_f(n)*\mu(n)\psi'(n)]\exp(-ny)  = R_0+\mathcal{R}(y)+ \frac{ 2}{\epsilon_\psi'}   \Gamma(k+a)  \left( \frac{y }{M}  \right)^a \left(\frac{i}{2\pi}\right)^{k+a} \\
			& \times \sum_{n=1}^{\infty}\frac{[a_f(n)*\mu_k(n)\overline{\psi'}(n)]}{n^{k+a}}  
			 \left[ {}_2F_1 \left(\frac{k+a}{2}, \frac{k+1+a}{2}; \frac{1+2a }{2}; -\frac{y^2}{4 M^2n^2\pi^2} \right)-(1-a) \right],
			\end{align*}
where $R_0$ and $\mathcal{R}(y)$ are defined as in Theorem \ref{Main theorem2}.  
\end{corollary}
At the end we have given a Table \ref{table},  which includes numerical evidences for this corollary.  
Now,  we state an asymptotic expansion for the Lambert series \eqref{main Lambert series} as an application of Theorem \ref{Main theorem2}.
\begin{corollary}\label{maincorollary}
Assume all notations that are defined in Theorem \ref{Main theorem2}.  
Then for $y \rightarrow 0^+$, we have
	\begin{align*}
	&\sum_{n=1}^{\infty}[a_f(n)\psi(n)*\mu(n)\psi'(n)]\exp(-ny)=R_0+ \mathcal{R}(y)+ \sum_{m=0}^{M'-1} B_{m, a} y^{2m+a} +O_{f,\psi,\psi'}(y^{2M'+a}) 
	\end{align*}
	where  $M'$ is any large positive integer and $B_{m, a}$'s are some explicit constants.
	Further, if $f$ is a normalized Hecke eigenform and $\chi, \psi$ and $\psi'$ are real,  then on the assumption of the generalized Riemann hypothesis and the simplicity hypothesis, we have 
	\begin{align*}
	y^{\frac{1}{2}}\sum_{n=1}^{\infty}[a_f(n)\psi(n)*\mu(n)\psi'(n)]\exp(-ny) & =  	y^{\frac{1}{2}} R_0+\sum_{n=1}^{\infty} r_n cos(\theta_n-t_n \log y)  \\
 &	+	\sum_{m=0}^{M'-1} B_{m, a} y^{2m+a+\frac{1}{2}} +O_{f,\psi,\psi'}(y^{2M'+a+\frac{1}{2}}).
	\end{align*}
Here $r_n\exp(i\theta_n)$ denote the polar representation of $2  L_f(\rho_n, \psi)\Gamma(\rho_n)(L'(\rho_n, \psi'))^{-1}$ with $n$-th non-trivial zero  of $L(s, \psi')$ in the upper critical strip is given by $\rho_n= s_n+it_n$.
	
\end{corollary}
	\begin{remark}
	The cosine functions in Corollary {\rm \ref{maincorollary}} suggests the oscillatory behavior of the Lambert series  $y^{\frac{1}{2}}\sum_{n=1}^{\infty}[a_f(n)\psi(n)*\mu(n)\psi'(n)]\exp(-ny)$  as $y \rightarrow 0^{+}$,  which is also consistent with the observation of Zagier for $a_{0}(y)$ \eqref{zagier}.
\end{remark}

\section{Proofs of main results}
We begin by collecting a few well-known results which will be essential for the proof of  the main theorem.  First,  we state an important asymptotic expansion of the gamma function,  which is popularly known Stirlng's formula.
\begin{lemma}\label{Stirling}
	For $s = \sigma + i T$ in a vertical strip $\alpha \leq \sigma \leq \beta$, 
	\begin{equation}\label{Stirling_equn}
	|\Gamma (\sigma + i T)| = \sqrt{2\pi} | T|^{\sigma - 1/2} e^{-\frac{1}{2} \pi |T|} \left(1 + O\left(\frac{1}{|T|}\right)  \right) \quad {\rm as} \quad |T|\rightarrow \infty.
	\end{equation}
\end{lemma}
\begin{proof}
	We refer \cite[p. 151]{IK} for the proof.
\end{proof} 
\begin{lemma} \label{dupliation}
	For any complex number $z$, we have 
	\begin{equation}\label{duplication}
	\Gamma(2z)=\frac{\Gamma(z)\Gamma(z+\frac{1}{2})\, 2^{2z}}{2\sqrt{\pi}}.
	\end{equation}
\end{lemma}
This is the duplication formula for the gamma function.  The next result gives an important bound for the inverse of the Dirichlet $L$-function.  
\begin{lemma}\label{bound_zeta(s)}
	Assume there exist a sequence of positive numbers $T$ with arbitrary large absolute value satisfying $|T-\Im(\rho)| >	\exp(-A |\Im(\rho)|/\log(|\Im(\rho)|))$ for every  non-trivial zero $\rho$ of $L(s, \psi')$, where $A$ is  some suitable positive constant. Then,
	\begin{equation*}
	\frac{1}{|L(\sigma + i T, \psi')|} < e^{B T},
	\end{equation*}
	for some suitable constant $0<B < \pi/4$.
\end{lemma}
\begin{proof}
The proof of this lemma can be found in Titchmarsh \cite[p.~219]{Titchmarsh}.  
\end{proof}
The next result says that any $L$-function associated to a cusp form can be bounded by a suitable polynomial in a vertical strip.  
\begin{lemma}\label{bound}
	In any vertical strip $\sigma_0 \leq \sigma \leq b $,
	there exist a  
	constant $C(\sigma_0)$, such that 
	\begin{equation*}
	|L_f(\sigma + iT, \psi ) | \ll |T|^{C(\sigma_0)} 
	\end{equation*}
	as $|T| \rightarrow \infty.$
\end{lemma}
\begin{proof}
	One can see this result in \cite[p. 97, Lemma 5.2]{IK}.
\end{proof}  
Next, we state the functional equation for the Dirichlet $L$-function. 
\begin{lemma}\label{FEdirichlet}
Let $\psi'$ be a Dirichlet character of Modulo $M$. Then the Dirichlet $L$-function $L(s,\psi')= \sum_{n=1}^{\infty} \frac{\psi'(n)}{n^s}$ analytically extends to the whole complex plane and satisfies the functional equation: 
$$\left(\frac{\pi}{M}\right)^{-\frac{s+a}{2}} \Gamma\left(\frac{s+a}{2}\right) L(s, \psi') =\frac{\epsilon_\psi'}{i^a \sqrt{M}} \left(\frac{\pi}{M}\right)^{-\frac{1-s+a}{2}} \Gamma\left(\frac{1-s+a}{2}\right) L(1-s, \overline{\psi'}).$$
\end{lemma}
\begin{proof}
We refer \cite{IK} for the proof. 
\end{proof}
Now we are ready to give the proof of the main theorem.  
\begin{proof}[Theorem {\rm \ref{Main theorem2}}][] 
First, we note that the Lambert series $$\sum_{n=1}^{\infty}[a_f(n)\psi(n)*\mu(n)\psi'(n)]\exp(-ny)$$ converges absolutely and uniformly for any $y>0$.  
Now using inverse Mellin transform for $\Gamma(s)$,  one can write
\begin{align}
 \sum_{n=1}^{\infty}[a_f(n)\psi(n)*\mu(n)\psi'(n)]\exp(-ny) = &\sum_{n=1}^{\infty}[a_f(n)\psi(n)*\mu(n)\psi'(n)]\frac{1}{2 \pi i} \int_{c-i \infty}^{c+ i \infty} \frac{\Gamma(s)}{ (ny)^{s}} \mathrm{d}s \nonumber \\
&= \frac{1}{2 \pi i} \int_{c-i \infty}^{c+ i \infty}   \frac{ \Gamma(s)  L_f(s, \psi)}{L(s,\psi')} y^{-s} {\rm d}s,  \label{firstintegral}
\end{align}
here interchange of summation and integration is possible only when  we start with $\Re(s)=c > \frac{k+1}{2}$. 
Next,  to simplify this line integral we shall take help of contour integration and use Cauchy's residue theorem.  
Consider the contour $\mathcal{C}_T$ determined by the line segments $[c-iT, c+iT], [c+i T, \lambda + i T], [\lambda + i T,  \lambda - i T]$, and $ [\lambda - i T, c-iT]$, where $T$ is some large positive real number and $-1< \lambda <0$.  Before using Cauchy's residue theorem,  let us identify the poles of the integrand function.  From \eqref{FE},  it follows that $\Gamma(s) L_f(s, \psi)$ has no poles since $\Lambda_f(s, \psi) $ is an entire function. 
Hence poles of the integrand are only due to the zeros of $L(s,\psi')$. Note that, if $\psi'$ is an even character of modulus $M >1$,  then $L(s, \psi')$ has trivial zeros at $0,  -2,  -4,  \cdots$.   And if $\psi'$ is an odd character,  then $L(s, \psi')$ has trivial zeros at $-1,  -3,  -5,  \cdots$. 
Again,  we know that the non-trivial zeros of $L(s, \psi')$ lie in the strip $0 < \Re(s) <1$.  Therefore,  applying Cauchy's residue theorem, we have
\begin{align}\label{CRT}
\frac{1}{2\pi i} \int_{\mathcal{C}_T}  \frac{ \Gamma(s) L_f(s, \psi)}{L(s,\psi')} y^{-s} {\rm d}s = \mathcal{R}_{T}(y)+R_0,
\end{align}
where $\mathcal{R}_{T}(y)$ denotes the residual function, which includes finitely many terms contributed by the non-trivial zeros $\rho$ of $L(s, \psi')$ with $|\Im(\rho)| < T$ and $R_0$ is the residue at $s=0$.
Now,  we can write 
\begin{equation}\label{roundequation}
 \int_{\mathcal{C}_T} \frac{ \Gamma(s) L_f(s, \psi)}{L(s,\psi')}  y^{-s} {\rm d}s=
 \left( \int_{c-i T}^{c+ i T} +\int_{c+i T}^{\lambda+ i T}+\int_{\lambda+ i T}^{\lambda- i T}+ \int_{\lambda- i T}^{c-i T}\right) \frac{ \Gamma(s) L_f(s, \psi)}{L(s,\psi')}  y^{-s} {\rm d}s.
\end{equation}
Next, utilizing Lemma \ref{Stirling},  \ref{bound_zeta(s)} and \ref{bound},  one can show that both of the horizontal integrals 
\begin{align*}
H_1(T, y):=  \frac{1}{2 \pi i} \int_{c+i T}^{\lambda + i T}   \frac{\Gamma(s)  L_f(s, \psi)}{L(s,\psi')} y^{-s} {\rm d}s, \,\,
H_2(T, y):=  \frac{1}{2 \pi i}  \int_{\Gamma(s)  \lambda - i T}^{c -iT}   \frac{ \Gamma(s)  L_f(s, \psi)}{L(s,\psi')} y^{-s} {\rm d}s, 
\end{align*}
tend to zero as $T \rightarrow \infty$ through those values of $T$ which satisfy $|T-\Im(\rho)| > \exp(-A |\Im(\rho)|/\log(|\Im(\rho)|))$.
Therefore, letting $T \rightarrow \infty$ in  \eqref{CRT}  and in view of  \eqref{firstintegral} and \eqref{roundequation},  we arrive at
\begin{equation}\label{secondintegral }
\sum_{n=1}^{\infty}[a_f(n)\psi(n)*\mu(n)\psi'(n)]e^{-ny} 
= \frac{1}{2 \pi i} \int_{\lambda-i \infty}^{\lambda+ i \infty}    \frac{\Gamma(s) L_f(s, \psi)}{L(s,\psi')} y^{-s} {\rm d}s+ \mathcal{R}(y)+R_0,
\end{equation}
where the contribution of the residual term $R_0$ will be taken into account only when $\psi'$ is an even character with modulus $M>1$.  Therefore,  we have 
\begin{equation}\label{R_0}
	R_0= \lim_{s \rightarrow 0} s \, \Gamma(s) \frac{L_f(s, \psi)}{L(s,\psi')} y^{-s} 
	= \begin{cases}
	\frac{L_f'(0, \psi)}{L'(0,\psi')},              & {\rm if } \, \, a=0,  M>1, \\
	          0, & {\rm otherwise}.
	\end{cases} 
\end{equation} 
The function $\mathcal{R}(y)$ is the sum of the residual terms coming from the non-trivial zeros $\rho$ of $L(s, \psi')$.  This term can be evaluated in the following way:
\begin{align}\label{R(y)}
\mathcal{R}(y)  = \sum_{\rho} \lim_{s \rightarrow \rho} (s - \rho)  \frac{ \Gamma(s) L_f(s, \psi)}{L(s,\psi')} y^{-s} 
= \sum_{\rho}  \frac{ \Gamma(\rho) L_f(\rho, \psi)}{L'(\rho,\psi')y^{-\rho} }, 
\end{align} 
where the summation runs over the non-trivial zeros $\rho$ of $L(s, \psi')$.   Here we note that we have used the assumption that all non-trivial zeros of $L(s,  \psi')$ are simple.  Even if we do not assume the simplicity hypothesis,  then also one can figure out this residual term.  
Now we shall try to evaluate the left vertical integral
 \begin{equation}\label{Vfirst}
V(y) := \frac{1}{2 \pi i} \int_{\lambda-i \infty}^{\lambda+ i \infty}  \frac{  \Gamma(s)  L_f(s, \psi)}{L(s,\psi')} y^{-s} {\rm d}s,
\end{equation} 
where $-1 < \lambda <0$.  The functional equation \eqref{FE} of $L_f(s, \psi)$ suggests that
\begin{align}\label{apply functional}
\Gamma(s) L_f(s, \psi) = i^k  \left(\frac{\sqrt{N}}{2 \pi}\right)^{k-2s} \frac{\epsilon_{\psi}^2 }{r} \chi(r) \psi(Q)  \Gamma(k-s) L_g(k-s, \overline{\psi}).
\end{align}
Again,  Lemma \ref{FEdirichlet} yields that
\begin{equation}\label{apply functional Dirichlet}
\frac{1}{L(s, \psi')} = \frac{i^a \sqrt{M}}{\epsilon_{\psi'}} \left(\frac{\pi}{M}\right)^\frac{1-2s}{2} \frac{\Gamma(\frac{s+a}{2})}{\Gamma(\frac{1-s+a}{2})} \frac{1}{L(1-s, \overline{\psi'})}.
\end{equation}
Now substituting \eqref{apply functional} and \eqref{apply functional Dirichlet} in \eqref{Vfirst} and simplifying, we get
\begin{equation*}
V(y)=  \frac{ \sqrt{\pi} i^{k+a} N^{\frac{k}{2}} \epsilon_{\psi}^2  \chi(r) \psi(Q)  }{ 2^k \pi^k r \epsilon_{\psi'}} \,\frac{1}{2 \pi i}   \int_{\lambda-i \infty}^{\lambda+ i \infty}   \frac{\Gamma(\frac{s+a}{2})   \Gamma(k-s)}{\Gamma(\frac{1-s+a}{2})} \frac{L_g(k-s, \overline{\psi})}{L(1-s, \overline{\psi'})}  \left(\frac{ N y}{4 \pi M}\right)^{-s} {\rm d}s. 
\end{equation*}
After a change of variable from $s$ to $k-s$, $V(y)$ takes the following form:
\begin{equation}\label{V(y) second form}
V(y)=  C_{k, \chi, \psi, \psi'}\,\frac{1}{2 \pi i}   \int_{k-\lambda-i \infty}^{k-\lambda+ i \infty}   \frac{\Gamma(s) \Gamma \left(\frac{k-s+a}{2}\right)   }{\Gamma\left(\frac{s-k+1+a}{2}\right)} \frac{L_g(s, \overline{\psi})}{L(s-k+1, \overline{\psi'})}    \left(\frac{ N y}{4 \pi M}\right)^{s} {\rm d}s,
\end{equation}
where 
\begin{equation}\label{Cchi}
 C_{k, \chi, \psi, \psi'} :=  \frac{ \sqrt{\pi} i^{k+a} \epsilon_{\psi}^2  \chi(r) \psi(Q)  }{  r \epsilon_{\psi'}} \left( \frac{2M}{y \sqrt{N}}  \right)^k.
\end{equation}
Note that the Dirichlet series expansion of $L_g(s, \overline{\psi})$ and $L(s-k+1, \overline{\psi'})$ are absolutely convergent on the line  $\Re(s) = k- \lambda $ since $k-\lambda > k$ as $-1 < \lambda < 0$.  Therefore,  on the line $\Re(s) = k- \lambda $,  with the help of the Dirichlet series expansion,  one can write
\begin{equation}\label{finaldirichlet}
\frac{L_g(s, \overline{\psi})}{L(s-k+1, \overline{\psi'})} = \sum_{n=1}^{\infty} \frac{a_{g}(n)  \overline{\psi}(n) * \mu_k(n)\overline{\psi'}(n)}{n^s}.
\end{equation}
Now,  plugging \eqref{finaldirichlet} in \eqref{V(y) second form} and then taking summation outside of integration,  we get
\begin{equation}\label{Vlast}
V(y)=  C_{k,\chi, \psi, \psi'} \,\sum_{n=1}^{\infty} [a_{g}(n)  \overline{\psi}(n)  * \mu_k(n)\overline{\psi'}(n)]  \frac{1}{2 \pi i}  \int_{k-\lambda-i \infty}^{k-\lambda+ i \infty}   \frac{\Gamma(s) \Gamma \left(\frac{k-s+a}{2}\right)   }{\Gamma\left(\frac{s-k+1+a}{2}\right)} \left(\frac{ N y}{4 \pi M n}\right)^{s} {\rm d}s. 
\end{equation}
To simplify $V(y)$ further,  we shall concentrate on the following integral:
\begin{equation}\label{una}
U_{n,a}(y) := \frac{1}{2 \pi i} \int_{k-\lambda-i \infty}^{k-\lambda+ i \infty} \frac{\Gamma(s) \Gamma \left(\frac{k-s+a}{2}\right)   }{\Gamma\left(\frac{s-k+1+a}{2}\right)} \left(\frac{ N y}{4 \pi M n}\right)^{s}  {\rm d}s,
\end{equation}
by doing a change of variable from $s$ to $2s$ and invoking Lemma \ref{dupliation},  it takes the shape as
\begin{equation}\label{U}
U_{n,a}(y)= \frac{1}{\sqrt{\pi}} \frac{1}{2 \pi i} \int_{\frac{k-\lambda}{2}- i \infty}^{\frac{k-\lambda}{2}+ i \infty} \frac{\Gamma(s) \Gamma\left(s+\frac{1}{2}\right) \Gamma \left(\frac{k+a}{2}-s\right)   }{\Gamma\left(s+\frac{1-k+a}{2}\right)}z^{s} {\rm d}s,
\end{equation}
where  $z= \left(\frac{ N y}{2 \pi M n}\right)^2$.  To make more comprehensible form for $U_{n,a}(y)$,  we must employ the definition of the Meijer $G$-function.  Unfortunately,  by analysing the poles of the integrand of $U_{n,a}(y)$ we can verify that the line of integration $\Re(s) = \frac{k-\lambda}{2}$ does not distinguish all the poles of $\Gamma\left( \frac{k+a}{2}-s  \right)$ from the poles of $\Gamma(s) \Gamma\left( s + \frac{1}{2}  \right)$.  Note that $\frac{k}{2} < \frac{k- \lambda}{2} < \frac{k+1}{2}$ as $ -1 < \lambda  <0$.  Thus,  we shift the line of integration $\Re(s) = \frac{k-\lambda}{2}$ to the line $\Re(s) = c'$ where $  \frac{k}{2} -1 < c' < \frac{k}{2}$.   Now we can see that the line of integration $\Re(s) = c'$ does separate the poles.  At this moment,  we construct a new rectangular contour $\mathfrak{C}$ joining the line segments $[c'-i T,  \frac{k-\lambda}{2} -i T],  \left[ \frac{k- \lambda }{2} - i T,  \frac{k- \lambda }{2} + i T \right],  \left[ \frac{k- \lambda }{2} + i T,  c'+ i T \right] ,  $ and $[ c'+ i T,  c'-i T]$ and employing Cacuchy's residue theorem,  we have
\begin{align}\label{2nd use_CRT}
\frac{1}{2\pi i} \int_{\mathfrak{C}} \frac{\Gamma(s) \Gamma\left(s+\frac{1}{2}\right) \Gamma \left(\frac{k+a}{2}-s\right)   }{\Gamma\left(s+\frac{1-k+a}{2}\right)}z^{s} {\rm d}s =  R_\frac{k}{2},
\end{align}
where $R_{ \frac{k}{2}}$ is the residue at $s=\frac{k}{2}$,  and it can be evaluated as
\begin{align}\label{residue at k by 2}
R_{\frac{k}{2}} = \begin{cases}
- \frac{\Gamma\left(\frac{k}{2} \right) \Gamma\left(\frac{k+1}{2}\right)}{ \Gamma\left(\frac{1}{2} + a\right)} z^{\frac{k}{2} } & {\rm if}\,\, a=0,  \\
0,  & {\rm if}\,\, a=1.
\end{cases}
\end{align}
Now using Stirling's formula for gamma function,  one can show that the contribution of the horizontal integrals vanish as $T \rightarrow \infty$.  Therefore, letting $T \rightarrow \infty$ in \eqref{2nd use_CRT},  we  have
\begin{align}\label{with residue at k by 2}
 \frac{1}{2 \pi i} \int_{\frac{k-\lambda}{2}- i \infty}^{\frac{k-\lambda}{2}+ i \infty}  \frac{\Gamma(s) \Gamma\left(s+\frac{1}{2}\right) \Gamma \left(\frac{k+a}{2}-s\right)   }{\Gamma\left(s+\frac{1-k+a}{2}\right)}z^{s} {\rm d}s & =   \frac{1}{2 \pi i} \int_{c'- i \infty}^{c'+ i \infty}  \frac{\Gamma(s) \Gamma\left(s+\frac{1}{2}\right) \Gamma \left(\frac{k+a}{2}-s\right)   }{\Gamma\left(s+\frac{1-k+a}{2}\right)}z^{s} {\rm d}s \nonumber \\ 
& + R_\frac{k}{2}.
\end{align}
Now utilizing the definition \eqref{MeijerG} of the Meijer $G$-function and verifying all the necessary conditions,  one can show that
\begin{align}\label{apply_Meijer G}
 \frac{1}{2 \pi i} \int_{c'- i \infty}^{c'+ i \infty}  \frac{\Gamma(s) \Gamma\left(s+\frac{1}{2}\right) \Gamma \left(\frac{k+a}{2}-s\right)   }{\Gamma\left(s+\frac{1-k+a}{2}\right)}z^{s} {\rm d}s = G_{2,2}^{\,1,2} \!\left(   \,\begin{matrix} 1 , \frac{1}{2} \\ \frac{k+a}{2}, \frac{1+k-a}{2}\end{matrix} \; \Big| z   \right).
\end{align}
At this situation,  we shall invoke Slater's theorem \eqref{Slater} to simplify Meijer $G$-function in terms of hypergeometric functions.  Thus after a significant simplification,  we obtain
\begin{align}\label{Use_Slater}
G_{2,2}^{\,1,2} \!\left(   \,\begin{matrix} 1 , \frac{1}{2} \\ \frac{k+a}{2}, \frac{1+k-a}{2}\end{matrix} \; \Big| z   \right) = z^{\frac{k+a}{2}} \frac{\Gamma\left( \frac{k+a}{2} \right) \Gamma\left( \frac{1+k+a}{2} \right)}{\Gamma\left(\frac{1}{2} + a\right) }  {}_2F_1 \left(\frac{k+a}{2}, \frac{k+1+a}{2}; \frac{1+2a }{2}; -z \right)  
\end{align}
Now substituting $z= \left(\frac{ N y}{2 \pi M n}\right)^2$ in \eqref{Use_Slater} and in view of \eqref{residue at k by 2}, \eqref{with residue at k by 2} and \eqref{apply_Meijer G},  the integral $U_{n,a}(y)$ becomes
\begin{align}\label{final U}
U_{n,a}(y)  & = \frac{2}{\sqrt{\pi}} \left(\frac{ N y}{2 \pi M n}\right)^{k+a} \frac{\Gamma(k+a)}{2^k} \left[ {}_2F_1 \left(\frac{k+a}{2}, \frac{k+1+a}{2}; \frac{1+2a }{2}; -z \right)  - (1-a) \right]. 
\end{align}
Plugging the above expression of $U_{n,a}(y)$ in \eqref{Vlast},  the final expression for the left vertical integral $V(y)$  reduces to
\begin{align} \label{Final_left vertical}
V(y)  & = 2 N^{\frac{k}{2}+a}  \Gamma(k+a)  \left( \frac{y}{M}  \right)^a \left(\frac{i}{2\pi}\right)^{k+a} \frac{\chi(r) \psi(Q) \epsilon_\psi^2}{r\epsilon_{\psi'}}  \sum_{n=1}^{\infty}\frac{[a_g(n)\overline{\psi}(n)*\mu_k(n)\overline{\psi'}(n)]}{n^{k+a}} \nonumber  \\
		& \quad \qquad \qquad \times	 \left[ {}_2F_1 \left(\frac{k+a}{2}, \frac{k+1+a}{2}; \frac{1+2a }{2}; -\frac{N^2y^2}{4 M^2n^2\pi^2} \right)-(1-a) \right].
\end{align}
Using the definition \eqref{hyper} of the hypergeometric series one can show that the above infinite series is indeed convergent.
Finally,  combining \eqref{secondintegral },  \eqref{R_0},  \eqref{R(y)} and \eqref{Final_left vertical},  we finish the proof of Theorem \ref{Main theorem2}. 
\end{proof} 

The proof of Corollary \ref{M=r=1},  \ref{Q=1,  M=r},  and \ref{Q=r=1} are immediate implications of Theorem \ref{Main theorem2}.  We left it for the readers to verify.  

\begin{proof}[Corollary  {\rm \ref{maincorollary}}][]
Making use of the asymptotic expansion of $_2F_1$,  for $y \rightarrow 0^+$,  we have
\begin{align}\label{double series}
{}_2F_1 \left(\frac{k+a}{2}, \frac{k+1+a}{2}; \frac{1+2a }{2}; -\frac{N^2y^2}{2^2M^2n^2\pi^2} \right)  & = \sum_{m=0}^{M'-1} \frac{\left(\frac{k+a}{2}\right)_m \left(\frac{k+1+a}{2}\right)_m}{\left(\frac{1+2a }{2}\right)_m m!}\left(\frac{-Ny}{2Mn\pi} \right)^{2m} \nonumber  \\
& + O_{f, \psi, \psi'}\left(\left(\frac{y}{n}\right)^{2M'}\right),
\end{align}
where $M'$ is any large positive integer.
Now employing \eqref{double series} in Theorem \ref{Main theorem2},  we get
\begin{equation}
\sum_{n=1}^{\infty}[a_f(n)\psi(n)*\mu(n)\psi'(n)]\exp(-ny)=R_0+\mathcal{R}(y)+ \sum_{m=0}^{M'-1} B_{m,  a} y^{2m+a} +O_{f, \psi, \psi'}(y^{2M'+a}),
\end{equation}
where $B_{m, a}$'s are constants can be evaluated by the following formula: For $m=0$, 
 $$B_{m,  a}=   2 a N^{k/2}  \Gamma(k+a)  \left( \frac{y N}{M}  \right)^a \left(\frac{i}{2\pi}\right)^{k+a} \frac{\chi(r) \psi(Q) \epsilon_\psi^2}{r\epsilon_{\psi'}}  
  \sum_{n=1}^{\infty}\frac{ \,[ a_g(n)\overline{\psi}(n)*\mu_k(n)\overline{\psi'}(n)]}{n^{k+a+2m}},$$ and for $m>1$,
 \begin{align}\label{cma}
 B_{m, a}   =  \frac{ 2  N^{k/2}  \Gamma(k+a)}{ (2\pi)^2m}  \left( \frac{y N}{M}  \right)^{a+2m}  \left(\frac{i}{2\pi}\right)^{k+a} & \frac{\chi(r) \psi(Q) \epsilon_\psi^2}{r\epsilon_{\psi'}}    \frac{( \frac{k+a}{2})_m(\frac{k+1+a}{2})_m\,}{(\frac{1+2a}{2})_m m!} \nonumber \\
 & \times \sum_{n=1}^{\infty}\frac{a_g(n)\overline{\psi}(n)*\mu_k(n)\overline{\psi'}(n)}{n^{k+a+2m}}.
\end{align}
Here we note that both of the  above infinite series are absolutely convergent.  
 Now we assume $f$ is a normalized Hecke eigenform and $\chi, \psi$ and $\psi'$ are real characters.  Then for any complex number $s$,  we get $\frac{L_f(\overline{s}, \psi)\Gamma(\overline{s})}{L'(\overline{s}, \psi')}\frac{1}{y^{\overline{s}}}=	\overline{\left(\frac{L_f(s, \psi)\Gamma(s)}{L'(s, \psi')}\frac{1}{y^s}\right)}$.  Another important observation is that if $\frac{1}{2}+i t_n$ is a non-trivial zero of $L(s,  \psi')$,  then $\frac{1}{2}-it_n$ is also a non-trivial zero of $L(s,  \psi')$ since $\psi'$ is a real character.  
  Therefore,  assuming generalized Riemann hypothesis and the simplicity hypothesis,  we can write
 $$\mathcal{R}(y)= \sum_{\rho_n=\frac{1}{2}+it_n, t_n>0} 2 \Re\left(\frac{L_f(\rho_n, \psi)\Gamma(\rho_n)}{L'(\rho_n, \psi')}y^{-\rho_n}\right),$$
where the sum is running over all non-trivial zeros of $L(s,\psi')$ in the upper critical line.
Finally, representing $2 L_f(\rho_n, \psi)\Gamma(\rho_n) {L'(\rho_n, \psi')}^{-1}$ in the polar form by $r_n\exp(i\theta_n)$ and simplifying we complete the proof of Corollary \ref{maincorollary}.
\end{proof}

\section{Final remarks}
In 2018,  Banerjee and Chakraborty \cite{BC} established an asymptotic expansion for the Lambert series $\sum_{n=1}^{\infty}\lambda_f(n) \overline{ \lambda_g(n)} \exp(-n y)$,  where $\lambda_f(n)$ and $\lambda_g(n)$ are $n$th Fourier coefficient of Hecke Mass cusp forms $f$ and $g$ respectively.  Recently,  the same Lambert series corresponding to the Fourier coefficents of Hilbert modular forms has been studied by Agnihotri \cite{Agni}. 
In this paper,  we have established an exact formula the Lambert series
 $A_f(y)=\sum_{n=1}^{\infty}[a_f(n)\psi(n)*\mu(n)\psi'(n)]\exp(-ny),$ in terms of the non-trivial zeros of $L(s, \psi')$,  where $a_f(n)$ is the $n$th Fourier coefficient of a cusp form $f$ over a congruence subgroup, and $\psi$ and $ \psi'$ are primitive Dirichlet characters.  Interestingly,  we have observed that $y^{1/2}A_f(y)$ has an oscillatory behaviour as $y \rightarrow 0^{+}$.  It would be desirable to study similar Lambert series associated to the $n$th Fourier coefficient of other automorphic forms.  

\begin{center}
\begin{table}[h]
\caption{Verification of Corollary \ref{Q=r=1}: Let $\psi'$ be a Dirichlet character modulo $5$ with $\psi'(\bar{1})=\psi'(\bar{4})=1$ and 
$\psi'(\bar{2})=\psi'(\bar{3})=-1$. 
We took $f(z)=\Delta(z)$ as Ramanujan delta function, and the left-hand side and right-hand side series over $n$ with only first $2000$ terms, and the  sum over $\rho$ for $\mathcal{R}(y)$ is taken over only $22$ terms.}
\label{table}
\renewcommand{\arraystretch}{1}
{\small
\begin{tabular}{|l|l|l|l|l|l|}
\hline
  $y$ & Left-hand side  & Right-hand side  \\
  \hline 
  $1.589$          & $0.02160533841$   &    $0.02160532545$ \\
  \hline
  $1+\sqrt{5}$     & $0.01599519746$   &    $0.01599520708$  \\      
\hline
 $0.0749$          & $0.03507904537$   &    $0.03507917507$  \\
\hline
$4-\pi$            & $0.01767636417$   &    $0.01767636262$ \\
\hline
  $\pi^{\sqrt{3}}$ & $0.00069009521$   &    $0.00069009799$ \\
\hline
   $5.7395$           & $0.00298669912$      & $0.00298669847$ \\
\hline
\end{tabular}}

\end{table}
\end{center}

\textbf{Acknowledgements.}
The first author wants to thank  SERB for the Start-Up Research Grant SRG/2020/000144. 
The second and third author wish to thank the National Institute of Technology Karnataka,  for providing conductive research environment.

\end{document}